\documentclass[11pt]{article}
\usepackage{amssymb}
\usepackage{amsmath}
\usepackage{bbm}
\usepackage{multirow}	
\usepackage{enumerate}
\usepackage{graphicx}
\usepackage{epstopdf}
\usepackage{color}
\usepackage{booktabs}
\usepackage[tight]{units}
\numberwithin{equation}{section}

\usepackage{amsthm}
\newtheorem{theorem}{Theorem}[section]

\newtheorem{definition}{Definition}[section]
\newtheorem{lemma}{Lemma}[section]
\newtheorem{corollary}{Corollary}[section]
\newtheorem{proposition}{Proposition}[section]

\newcommand{\curl}{\mathop{\mathbf{curl}}\nolimits}
\newcommand{\dive}{\mathop{\mathrm{div}}\nolimits}
\newcommand{\dif}{\overline{\partial}}

\newcommand{\norm}[1]{\left\lVert#1\right\rVert}
\newcommand{\set}[1]{\left\{#1\right\}}
\newcommand{\produ}[1]{\left\langle#1\right\rangle}
\newcommand{\abs}[1]{\left\lvert#1\right\rvert}

\newcommand{\X}{X}
\newcommand{\Y}{Y}
\newcommand{\R}{\mathbb{R}}
\newcommand{\oA}{\hat{\Omega}}
\newcommand{\ocA}{\hat{\Omega}_{\mathrm{c}}}
\newcommand{\occA}{\overline{\hat{\Omega}}_{\mathrm{c}}}
\newcommand{\oc}{\Omega_{\mathrm{c}}}
\newcommand{\occ}{\overline{\Omega}_{\mathrm{c}}}

\renewcommand{\H}{\mathrm{H}}
\renewcommand{\L}{\mathrm{L}}

\newcommand{\huno}{\H^1(\Omega)}

\newcommand{\ldoso}{\L^2(\Omega)}

\newcommand{\ldosoc}{\L^2(\oc)}

\newcommand{\cero}{\boldsymbol{0}}

\newcommand{\dt}{\,dt}

\newcommand{\bw}{\mathbf{w}}
\newcommand{\xn}{\mathbf{x}}

\newcommand{\Jn}{\mathbf{J}}
\newcommand{\Hn}{\mathbf{H}}

\newcommand{\pct}{\mathrm{a.e.}}

\newcommand{\V}{X}
\newcommand{\Z}{Z}

\newcommand{\Hm}{\boldsymbol{H}}
\newcommand{\Em}{\boldsymbol{E}}

\newcommand{\ee}{\mathbf{E}}
\newcommand{\hh}{\Hn}
\newcommand{\jj}{\Jn}
\newcommand{\deriparc}[2]{\dfrac{\partial#1}{\partial#2}}
\newcommand{\grad}{\nabla}
\newcommand{\jd}{J_{\mathrm{d}}}
\newcommand{\paren}[1]{\left(#1\right)}
\newcommand{\supp}{\mathop{\mathrm{Supp}}\nolimits}
\newcommand{\houno}{\H_0^1(\Omega)}
\newcommand{\ldoshouno}{\L^2(0,T;\houno)}
\newcommand{\ldosto}{\L^2(0,T;\ldoso)}
\newcommand{\uhn}{u_h^n}
\newcommand{\uhnm}{u_h^{n-1}}

\newcommand{\dual}[2]{\left\langle{#1},{#2}\right\rangle}
\newcommand{\wunodos}{\mathrm{W}^{1,2}(0,T;X,X')}

\newcommand{\ldoshounop}{\L^2(0,T;\houno')}
\newcommand{\hunooc}{\H^1(\oc)}

\newtheorem{problem}{{\bf Problem}}
\newtheorem{remark}{{\bf Remark}}

\def\thanksramiro{Universidad del Cauca, Popay\'an, Colombia, 
 email: {\tt rmacevedo@unicauca.edu.co}.}
                  
\def\thankschristian{Universidad Nacional de Colombia sede Medell\'in, Colombia
                  email: {\tt chcgomezmo@unal.edu.co}., }  
\def\thanksbibiana{Universidad Nacional de Colombia sede Medell\'in, Colombia
                  email: {\tt blopezr@unal.edu.co}., }  
\begin{document}
\title{Fully-discrete finite element approximation for a family of degenerate parabolic problems}



\author{
{\sc Ramiro Acevedo}\thanks{\thanksramiro}\quad
{\sc Christian G\'omez}\thanks{\thankschristian}\quad  
{\sc Bibiana L\'opez-Rodr\'{\i}guez}\thanks{\thanksbibiana}
}

\date{Received: date / Accepted: date}

\maketitle

\begin{abstract}
The aim of this work is to show an abstract framework to analyze the numerical approximation 
by using a finite element method in space and a Backward-Euler scheme in time of a family of degenerate parabolic problems. 
We deduce sufficient conditions to ensure that the fully-discrete problem has a unique solution and to prove quasi-optimal error estimates for the approximation.
Finally, we show a degenerate parabolic problem which arises from electromagnetic applications and deduce its well-posedness and convergence by using the developed abstract theory, including numerical tests to illustrate the performance
of the method and confirm the theoretical results.
\\[2ex]
\textbf{Keywords:}
{parabolic degenerate equations, parabolic-elliptic equations, finite element method, backward Euler scheme, fully-discrete approximation, error estimates, eddy current model.}
\end{abstract}

\section{Introduction}
\label{intro}

A \textit{degenerate parabolic equation} \cite[Chapter III]{showalter} (also called \textit{parabolic-elliptic equation} \cite{pluschke})
is an abstract evolution 
equation of the form
\begin{equation}\label{deg-1}
\dfrac{d}{dt}(Ru(t))+ A(t)u(t)=f(t),
\end{equation}
where $R$ is a linear, bounded and monotone operator 
and $(A(t))_{t\in[0,T]}$ is a family of linear and bounded operators. 
They arise in several applications, for instance 
in the study of eddy currents in electromagnetic field theory
(see \cite{zlamal,maccamy,reales}). 


Results about existence and uniqueness of solutions for some degenerate 
parabolic equations have been widely studied. In \cite{kuttler} Kuttler \& Kenneth L. 
show results concerning existence, uniqueness and regularity of equations 
of the form \eqref{deg-1}, but with R non-invertible and A a linear operator independent 
of the time. 
Sufficient conditions to ensure the 
existence and uniqueness of solutions of \eqref{deg-1}, 
even when R depends on the time, are shown by Showalter \cite{showalter} 
(see also \cite{showalter-2}). Moreover, the existence and uniqueness of the solutions
for the case of the family of operators $A$
can be non-linear, has been analyzed in \cite{kuttler-1982, kuttler-1986, paronetto}. 

Among the numerical methods found in the literature
to compute the approximated solution 
of classical parabolic partial differential equation, the finite element method
(with some time-stepping scheme) is one of the more extended. 
We can cite
the book by V. Thom\'{e}e \cite{thomee}
as a classical reference about this topic. Moreover, books dedicate to
the finite element approximation for partial differential equations, devote at least 
one chapter to the analysis of the numerical approximation of parabolic equations
(see, for instance, \cite{guermond} and \cite{quarteronivalli}). In fact, the 
developed theory for the approximation of parabolic equations by the finite element method, 
is mainly presented 
for a general heat-like equation, i.e., to approximate the solution of a general parabolic 
problem of the 
form: 
\begin{equation}\nonumber
\dfrac{du}{dt} + \mathcal{L}u= f,
\end{equation}
with $\mathcal{L}$ is a coercive differential operator of the second order.

The mathematical analysis for the numerical approximations by finite element methods, 
including existence and uniqueness of the discrete solutions and quasi-optimal error estimates,
has been only performed for particular degenerate parabolic equations. For instance, 
Zlamal \cite{zlamal} has studied the approximation of solution for a two-dimensional eddy current 
problem in a bounded domain, MacCamy \& Zuri \cite{maccamy} 
have proposed a FEM-BEM coupling for the formulation analyzed in \cite{zlamal}, and  
a formulation for an axisymmetric eddy current problem was studied by
Bermudez \textit{et al} \cite{reales}. The formulations studied in all these references can be expressed as
particular cases of problem \eqref{deg-1}. Nevertheless, to the best knowledge of the authors, 
there is not an abstract general theory that allows to deduce the mathematical analysis of
these approximations as particular applications of that theory. 

The main goal of this article is precisely to provide a general theory for the mathematical analysis of a
fully-discrete finite element approximation
for an abstract degenerate parabolic 
equation. To this aim, we consider a fully discrete approximation for a Cauchy problem 
associated to equation \eqref{deg-1}, 
by using a finite element method in space and a Backward-Euler scheme in time. 
We show sufficient conditions 
for the spaces and the family of operators, to guarantee existence and uniqueness of
the fully-discrete solutions by assuming that the time step is sufficiently small. 
Furthermore, we prove quasi-optimal  error estimates for this 
fully discretized scheme by adapting the approximation theory 
for classical parabolic equations to the abstract degenerate case.
Moreover, since a good discrete approximation for the time-derivative of the solution is relevant for the applications,
we prove that this time derivative
can be approximated with quasi-optimal error estimates. 

The outline of the paper is as follows: Section~\ref{sec:1} is devoted to show some concepts about spaces for evolutive problems and the abstract framework for degenerate parabolic equations and their well-posedness are recalled in Section~\ref{degenerate}.  The corresponding analysis for the fully-discrete approximation of problem 
by using finite element method in space and a backward Euler scheme in time, is presented in Section \ref{discreto-d} and the results ensuring the quasi-optimal 
convergence of the approximation method are shown in Section~\ref{error-d}. 
Furthermore,  the application of the theory to an eddy current model is studied in Section \ref{aplicaciones-d}, where we
deduce its well-posedness and theoretical convergence by using the developed abstract theory.
Finally, we show some numerical results that confirm the expected convergence of the method according to the theory.

\section{Hilbert functional spaces for evolutive problems} \label{sec:1}
Let us first review some basic concepts about functional analysis which are useful in dealing with time-dependent functions. A complete and detailed presentation of the concepts that we indicate in this section can be founded, for instance, in  \cite[ Sections 23.2-23.6]{zeidlerl}. More precisely, we need to introduce spaces of functions defined on a bounded time interval $(0,T)$ (where $T>0$ is a fixed time) and with values in separable Hilbert space $X$. We will denote by $\|\cdot\|_{X}$, $(\cdot,\cdot)_{X}$ and $\langle \cdot,\cdot\rangle_{X}$, the norm, the inner product and duality pairing in $\X$. We use the notation $\mathcal{C}^0([0,T];X)$ for the space consisting of all continuous functions $f:[0,T]\to X$. More generally, for any $k\in\mathbb{N}$, $\mathcal{C}^k([0,T];X)$ denotes the subspace of $\mathcal{C}^0([0,T];X)$ of all functions $f$ with (strong) derivatives of order at most $k$ in $\mathcal{C}^0([0,T];X)$, i.e., 
\[
\mathcal{C}^k([0,T];X)
:=\set{f\in\mathcal{C}^0([0,T];X):\quad
\frac{d^jf}{dt^j}
\in\mathcal{C}^0([0,T];X),\quad 1\leq
j\leq k}.
\]
A classical result of functional analysis states  $\mathcal{C}^k([0,T];X)$ is a Banach space with the norm
\[
\norm{f}_{\mathcal{C}^k([0,T];X)}:=\sup_{t\in[0,T]} \sum_{j=0}^k\norm{\frac{d^j f}{dt^j}(t)}_{X}.
\]

We also consider the space  $\L^2(0,T; X)$ of classes of functions $f:(0,T)\to X$ that are B\"ochner-measurable whose norm in $X$ belongs to   $\L^2(0,T)$, i.e., 
\[
\norm{f}^2_{\L^2(0,T; X)}:= \int_0^T\norm{f(t)}_{X}^2 \dt <
+\infty.
\]
The space $\L^2(0,T;X)$ is a Hilbert space with the norm $\|\cdot\|_{\L^2(0,T;X)}$. Furthermore, the dual space of $\L^2(0,T;X)$ can be identified with the space $\L^2(0,T;X')$ as shown in the following result. 

\begin{proposition}[Dual space of $\L^2(0,T;X)$]
Let $X$ be a separable Hilbert space. 
For any $f\in \L^2(0,T;X)'$ there exists a unique $v_f\in\L^2(0,T; X')$ satisfying
\[
\dual{f}{w}=\int_0^T\dual{v_f(t)}{w(t)}_{X}dt\qquad\forall w\in \L^2(0,T;X).
\]
Moreover, the map $f\mapsto v_f$ is a linear bijection which preserves the norm, i.e., 
\[
\norm{f}_{\paren{L^2(0,T;X)}'}=\norm{v_f}_{L^2(0,T;X')}\qquad\forall f\in\paren{L^2(0,T;X)}'.
\]
\end{proposition}
\begin{proof}
See, for instance, \cite[Proposition~23.7]{zeidlerl}.
\end{proof}

The analysis of evolutive differential problems require functional spaces involving time-derivatives. Let $\X$ and $\Y$ be two separable Hilbert spaces such that $X\subset Y$ with continuous and dense embedding. Let $X'$ the dual space of $X$ with respect to the pivot space $Y$. More precisely, $Y$ can be identified as a subset of $X'$ and 
\[
\langle w, v\rangle_{X}=(w,v)_{X}\quad \forall w\in Y \quad \forall v\in X. 
\]

We will denote by $\wunodos$ the functional space given by
\[
\wunodos:=
\left\{
v\in \L^2(0,T;X):\: \dfrac{dv}{dt}\in \L^2(0,T;X')
\right\}
\]
where $\dfrac{dv}{dt}$ is the \textit{generalized time-derivative} of $v$ characterized by 
\[
\int_{0}^T\dual{\dfrac{dv}{dt}(t)}{w}_{X}\varphi(t)dt= - \int_{0}^T\paren{v(t),w}_X\varphi'(t)dt \qquad \forall w\in X \quad\forall\varphi\in C_0^{\infty}(0,T).
\]
It is well known that $\wunodos$ endowed with the norm
\[
\norm{v}_{\wunodos}:=\norm{v}_{L^2(0,T;X)} + \norm{\dfrac{dv}{dt}}_{L^2(0,T;X')}
\]
is a Banach space and $\wunodos\subset \mathcal{C}^0([0,T];Y)$ with a continuous embedding (see, for instance, \cite[Proposition~23.23]{zeidlerl}). 

Let $k\in\mathbb{Z}^+$. The generalized time-derivative of order $k$ of $v\in \L^2(0,T;X)$, denoted by $\dfrac{d^kv}{dt^k}$, can be defined inductively. 
Hence, we can consider the space
\[
\mathrm{H}^k(0,T;X):=\left\{
v\in \L^2(0,T;X):\: \dfrac{d^jv}{dt^j}\in \L^2(0,T;X),\: j=1,\ldots,k
\right\},
\]
which is a Banach space with the norm
\[
\norm{v}_{\mathrm{H}^k(0,T;X)}:=\sum_{j=0}^k\norm{\dfrac{d^jv}{dt^j}}_{L^2(0,T;X)}.
\]
Furthermore, the embedding 
$
\mathrm{H}^k(0,T;X) \subset \mathcal{C}^{k-1}([0,T];X)
$
is continuous for any $k\in\mathbb{Z}^+$. 

\section{The degenerate parabolic problem}\label{degenerate}

Let $\X$ and $\Y$ be two real separable Hilbert spaces such that $\X\subset\Y$ with continuous and dense embedding. We denote by $(\cdot,\cdot)_{\V}$ and $(\cdot,\cdot)_{\Y}$ the inner products on $\X$ and $\Y$ respectively and $\|\cdot\|_{\V}$, $\|\cdot\|_{\Y}$ the corresponding norms. Furthermore, $\langle\cdot,\cdot\rangle_{\V}$ and $\langle\cdot,\cdot\rangle_{\Y}$ denote respectively the duality paring of $\V$ and $\Y$ and their corresponding dual spaces. Let $R:\Y \to \Y'$ a linear and bounded operator.
Let $T>0$, for any $t\in[0,T]$, let us consider a linear and bounded operator $A(t):\V\to\V'$. 
Then, given $f\in \L^2(0,T;\V')$ and $u_{0}\in \Y$, the degenerate parabolic problem can read as follows.

\begin{problem}
\label{ppdc}
Find $u\in\L^2(0,T;\V)$ such that:
\begin{align*}
\frac{d}{dt}\left\langle Ru(t),v\right\rangle_{\Y} + \produ{A(t)u(t),v}_{\X}=&\produ{f(t),v}_{\X}\quad \forall v\in \V ,\\
\left<Ru(0),v\right>_{\Y}=&\left<Ru_0,v\right>_{\Y}\quad\forall v \in \Y.
\end{align*}
\end{problem}

The first identity in Problem~\ref{ppdc} is given in the space of the distributions $\mathcal{D}'(0,T)$, i.e., this equation is equivalent to
\begin{equation}\nonumber
- \int_0^T \left\langle Ru(t),v\right\rangle_{\Y} \varphi'(t)dt  + \int_0^T  \produ{A(t)u(t),v}_{\X} \varphi(t)dt =  \int_0^T  \produ{f(t),v}_{\X} \varphi(t)dt
\end{equation}
for all $v\in \V$ and $\varphi\in C_0^\infty(0,T)$. Moreover, Problem~\ref{ppdc} can be formulated as any of the following two equivalent problems.
\begin{problem}
\label{ppdint}
Find $u\in \L^2(0,T;\V)$ such that 
\[
-\int_0^T{\left\langle Ru(t), v'(t)\right\rangle}_{\Y}dt+\int_0^T{\left\langle A(t)u(t),v(t)\right\rangle}_{\V}dt
=\int_0^T{\left\langle f(t),v(t)\right\rangle}_{\V}dt+\left\langle Ru_0,v(0)\right\rangle_{\Y},
\]
for all  $v\in \L^2(0,T;\V)\cap H^1(0,T;\Y)$ with $v(T)=0$.
\end{problem}

\begin{problem}
\label{ppddual}
Find $u\in \L^2(0,T;\V)$ satisfying
\begin{align*}
\frac {d}{dt}Ru(\cdot) +A(\cdot)u(\cdot)&=f(\cdot)\quad\text{in} \ \  \L^2(0,T;\V'),\\
Ru(0)&=Ru_0\quad \text{in}\quad \Y'.
\end{align*}
\end{problem}

Let us remark that the first equation in Problem~\ref{ppddual} implies that  $Ru(\cdot)\in H^1(0,T;\V')$, consequently the function $t\mapsto Ru(t)$ is absolutely continuous in $\V'$ and, in particular, $Ru(0)\in\V'$. On the other hand, since the inclusion $X\subset Y$ is dense and continuous, the inclusion $Y'\subset X'$ is also dense and continuous and therefore, by recalling that $Ru_0\in \Y'$, the initial condition given by the second equation of Problem~\ref{ppddual} has meaning, which is equivalent to the second equation of Problem~\ref{ppdc}.

In order to obtain the well-posedness result for Problem~\ref{ppdc} (and equivalently for Problem~\ref{ppdint} and Problem~\ref{ppddual}), we need to recall the following definition; see \cite[Section~III.3]{showalter}.
\begin{definition}
Let $\Z$ be a real separable Hilbert space and $\mathcal{G}:=\{G(t):\Z\to \Z':t\in[0,T]\}$ be a family of linear and bounded operators.  $\mathcal{G}$ is called \textit{monotone}, if $\left<G(t)v,v\right>_{\Z}\geq 0$ for any $v\in \Z$ and for any $t\in[0,T]$. $\mathcal{G}$ is called \textit{self-adjoint}, if $\left<G(t)u,v\right>_{\Z}=\left<G(t)v,u\right>_{\Z}$ for any $u,v\in \Z$ and for any $t\in[0,T]$. Similarly, $\mathcal{G}$ is called \textit{regular} if for each $u,v\in \Z$ the map $t\mapsto \left<G(t)u,v\right>_{\Z}$ is absolutely continuous on $[0,T]$ and there exists a function $k:(0,T)\to\mathbb{R}$ belongs to $L^1(0,T)$, which satisfies
\[
\abs{\frac{d}{dt}\left<G(t)u,v\right>_{\Z}}\leq k(t)\|u\|_\Z \|v\|_\Z\qquad \forall u,v\in \Z\quad\textrm{a.e. } t\in [0,T].
\]
\end{definition} 

The following result shows sufficient conditions to obtain the existence and uniqueness of solution for Problem~\ref{ppdc} and its proof can be founded in \cite[Proposition~III.3.2 and III.3.3]{showalter}. 
\begin{theorem}\label{welld}
Assume that the operator $R$ is monotone, self-adjoint, and there exist constants $\lambda>0$ and $\alpha>0$ such that
\begin{equation}\label{existencia}
\lambda\left<Rv,v\right>_{\Y}  + \left<A(t)v,v\right>_{\V}
\geq \alpha \|v\|_\V^2\quad \forall v\in \V\quad\forall t\in [0,T]. 
\end{equation}
Then, there exists a solution of Problem~\ref{ppdc} and it satisfies
\begin{equation}\label{dependencia-d}
\|u\|_{\L^2(0,T;\V)}\leq C \left(\|f\|_{\L^2(0,T;\V')}^2+ \left<Ru_0,u_0 \right>_{\Y}\right)^{\frac{1}{2}},
\end{equation}
for some constant $C>0$. Furthermore, if $A(t)$ is a regular family of self-adjoint operators, then the solution of Problem~\ref{ppdc} is unique.
\end{theorem}

\section{Fully-discrete approximation for degenerate parabolic problem}\label{discreto-d}
In this section we present the fully-discrete approximation for the degenerate parabolic problem which was introduced in the previous section. To this aim, we assume that the family of operators $A(t)$ and the operator $R$ satisfy the sufficient conditions given in Theorem~\ref{welld} to guarantee the existence and uniqueness of solution of Problem~\ref{ppdc}. 

The fully-discrete approximation will be obtained by using the finite-element method in space and a backward-Euler scheme in time. Let $\{\X_h\}_{h>0}$ be a sequence of finite-dimensional subspaces of $\X$ and let
$t_n: =n\Delta t$, $n=0,\ldots,N,$ be a uniform partition of $[0,T]$ with a time-step $\Delta t := T/N$. 

For any finite sequence $\{\theta^n: \ n=0,\ldots,N\}$ 
we denote
\[
\dif\theta^n:=\frac{\theta^n-\theta^{n-1}}{\Delta t},\qquad n=1,\dots,N.
\]

Let $u_{0,h}\in\X_h$ a given approximation of $u_0$. The fully-discrete approximation of Problem~\ref{ppdc} reads as follows.
\begin{problem}\label{ppdd} 
Find $u_h^n\in\V_h$, $n=1,\ldots,N$, such that
\begin{align*}
\langle R\dif u_h^n ,v\rangle_{\Y} +\langle A (t_n)u_h^n,v\rangle_{\V}
&=\langle f(t_n),v\rangle_{\V}\qquad \forall v\in\V_h.\\
u_h^0&=u_{0,h}
\end{align*}
\end{problem}

We can easily check that in each step $n=1,\ldots,N$, $u_h^n$ is computed as the solution of the following problem: find $u_h^n\in\V_h$ such that
\begin{align*}
\mathcal{A}_{n}(u_h^n,v)=F_n(v)\qquad \forall v\in \V_h,
\end{align*}
where $\mathcal{A}_{n}$ and $F_n$ are defined by
\begin{align*}
\mathcal{A}_{n}(w,v)&:=\langle R w,v\rangle_{\Y}+\Delta t \ \langle A(t_n)w,v \rangle_{\V}\qquad\forall w,\ v\in\V_h,\\
F_n(v)&:=\Delta t\ \langle f(t_n),v\rangle_{\V}+\langle R u_h^{n-1},v\rangle_{\Y}\qquad\forall v\in\V_h. 
\end{align*}
We will use the Lax-Milgram Lemma to deduce the existence and uniqueness of solution of Problem~\ref{ppdd} for each $n=1,\ldots,N$. Since $F_n$ is linear and bounded and $\mathcal{A}_{n}$ is bilinear and bounded, we need to prove that $\mathcal{A}_{n}$ is elliptic in $\V_h$. In fact, if we assume that $0<\Delta t\leq 1/\lambda$, for any $v\in\V_h$ we have
\[
\mathcal{A}_{n}(v,v)=\langle R v,v\rangle_{\Y}+\Delta t  \langle A(t_n)v,v \rangle_{\V}
\geq \Delta t\left[\lambda\langle R v,v\rangle_{\Y}+\produ{A(t_n)v,v}_{\V}\right],
\]
then, from \eqref{existencia} it follows that 
\[
\mathcal{A}_{n}(v,v)\geq \alpha\Delta t \|v\|^2_{\V}\qquad \forall v\in \V_h.
\]
Consequently, we have the following result about the existence and uniqueness of solution for the fully-discrete Problem~\ref{ppdd}. 
\begin{theorem}\label{welldh}
Assume that the family of operators $A(t)$ and the operator $R$ satisfy the sufficient conditions given in Theorem~\ref{welld} to guarantee the existence and uniqueness of solution of Problem~\ref{ppdc}. If the time-step $\Delta t$ is small enough (\textit{e.g.}, $0<\Delta t\leq 1/\lambda$), the fully-discrete Problem~\ref{ppdd} has a unique solution $u_h^n\in\V_h$ for each $n=1,\ldots,N$.
\end{theorem}
\section{Error estimates for the fully-discrete approximation}\label{error-d}
In this section, we will deduce some error estimates for the fully-discrete approximation. To this aim, from now on we assume the assumptions of Theorems~\ref{welld} and \ref{welldh}. Moreover, we assume that the solution to Problem~\ref{ppdc} satisfies $u\in\H^1(0,T;\X)$. Furthermore, we consider the orthogonal projection operator $\Pi_h:\V \to\V_h$, defined by
\[
\Pi_h w\in\V_h:\quad (\Pi_h w,v)_{\V}=(w,v)_{\V}\quad \forall v\in \V_h,
\]
clearly, $\Pi_h$ is well-defined and satisfies
\begin{align}\label{infprodg}
\|w-\Pi_h w\|_{\V}\leq \inf_{v\in \V_h} \|w-v\|_{\V}\quad \forall w\in \V.
\end{align}

{}From now on $u$ and $u_h^n$, $n=1,\ldots,N$, denotes the solutions to Problem~\ref{ppdc} and Problem~\ref{ppdd}, respectively.
We define the error and consider its splitting
\begin{align}
\label{splitE-d}
e^n_h:=u(t_n)-u_h^n=\rho_h^n+\sigma_h^n, \qquad n=1,\ldots,N,
\end{align}
where
\begin{align}\label{rhoysigmad}
\rho_h(t):=u(t)-\Pi_h u(t),\quad \rho_h^n:=\rho_h(t_n),\quad \sigma_h^n:=\Pi_h u(t_n)-u_h^n.
\end{align}
Furthermore, we denote
\begin{equation}\nonumber
\tau^n:=\frac{u(t_n)-u(t_{n-1})}{\Delta t}-\partial_t u(t_n).
\end{equation}

\begin{lemma}\label{lemdg}
If $u\in\H^1(0,T;\X)$ then there exists a constant $C>0$, independent of $h$ and $\Delta t$, such that
\begin{equation}\label{desi-sigma}
\langle R\sigma_h^n,\sigma_h^n\rangle_{\Y}
+\Delta t\,\sum_{k=1}^n\|\sigma_h^k\|^2_{\X}
\leq C\left[\langle R\sigma_h^0,\sigma_h^0\rangle_{\Y}
+\Delta t\,\sum_{k=1}^N \left\{\|\tau^k\|^2_{\Y}+\|\dif\rho_h^k\|^2_{\X}+\|\rho_h^k\|^2_{\X}
\right\}\right].
\end{equation}
Furthermore, if $u_0\in\X$ and for each $t\in[0,T]$ the operator $A(t)$ is monotone and there exists a constant $C>0$ such that 
\begin{equation}\label{Ap-bounded}
\langle A'(t)u,v\rangle\leq C\|u\|_{\X}\|v\|_{\X}\quad \forall u,v\in \X\quad \forall t\in [0,T],
\end{equation}
then, there exists a constant $C>0$, independent of $h$ and $\Delta t$, such that 
\begin{equation}\label{estimasumaderivada-R}
\begin{split}
 \Delta t\sum_{k=1}^n\langle R\dif \sigma_h^k,\dif \sigma_h^k\rangle_{\Y}
&+\langle A(t_n)\sigma_h^n,\sigma_h^n\rangle_{X}\\ 
&\leq  
C\left[ \Vert \sigma_h^{0}\Vert_X^2 + \Vert \rho_h^{0}\Vert_X^2 + \Vert \rho_h^{n}\Vert_X^2
+\Delta t\sum_{k=1}^{N} \left\{\Vert{\tau^k}\Vert_{\Y}^2 + \Vert\dif\rho_h^k\Vert_{\X}^2+ \Vert\rho_h^k\Vert_{\X}^2\right\}
\right].
\end{split}
\end{equation}
\end{lemma}

\begin{proof}
Let $n\in\{1,\dots,N\}$, $k\in\{1,\dots,n\}$ and $v\in\V_h$. Then, from Problem~\ref{ppdc} and Problem~\ref{ppdd}, it follows that
\begin{equation}\label{indentityv}
\langle R\dif \sigma_h^k,v\rangle_{\Y}
+\langle A(t_k)\sigma_h^k,v\rangle_{\X}
=\langle R\tau^k,v\rangle_{\Y}
-\langle R\dif\rho_h^k,v\rangle_{\Y}
-\langle A(t_k)\rho_h^k,v\rangle_{\X}\qquad\forall v\in \V_h.
\end{equation}

By testing this previous identity with $v=\sigma_h^k\in \V_h$, we have
\begin{equation}\label{identitysigma}
\langle R\dif \sigma_h^k,\sigma_h^k\rangle_{\Y}
+\langle A(t_k)\sigma_h^k,\sigma_h^k\rangle_{\X}
=\langle R\tau^k,\sigma_h^k\rangle_{\Y}
-\langle R\dif\rho_h^k,\sigma_h^k\rangle_{\Y}
-\langle A(t_k)\rho_h^k,\sigma_h^k\rangle_{\X}.
\end{equation}
Using the fact that $R$ is monotone and self-adjoint, the first term of the left-hand term in the previous identity satisfies
\begin{equation*}
\langle R\dif \sigma_h^k,\sigma_h^k\rangle_{\Y} 
\geq  \dfrac{1}{2\Delta t} \left\{\langle R\sigma_h^k,\sigma_h^k\rangle_{\Y}  
- \langle R\sigma_h^{k-1},\sigma_h^{k-1}\rangle_{\Y}
\right\},
\end{equation*}
by recalling \eqref{existencia}, there exist $\lambda,\alpha>0$ such that
\[
\langle A(t_k)\sigma_h^k,\sigma_h^k\rangle_{\X} 
\geq \alpha \Vert{\sigma_h^k}\Vert_{\V}^{2} 
-\lambda \langle R\sigma_h^k,\sigma_h^k\rangle_{\Y}, 
\]
thus, replacing in \eqref{identitysigma}, it follows that
\begin{equation}\label{pre1B}
\begin{split}
\dfrac{1}{2\Delta t} \left[\langle R\sigma_h^k,\sigma_h^k\rangle_{\Y}  
- \langle R\sigma_h^{k-1},\sigma_h^{k-1}\rangle_{\Y}\right] 
&+ \alpha \Vert{\sigma_h^k}\Vert_{\V}^{2} 
-\lambda \langle R\sigma_h^k,\sigma_h^k\rangle_{\Y}\\
&\leq\langle R\tau^k,\sigma_h^k\rangle_{\Y}
-\langle R\dif\rho_h^k,\sigma_h^k\rangle_{\Y}
-\langle A(t_k)\rho_h^k,\sigma_h^k\rangle_{\X}
\end{split}
\end{equation}
Now, since the operator $R$ is monotone and self-adjoint, it satisfies the following Cauchy-Schwarz type inequality
\begin{equation}\label{cauchy-R2}
\vert\langle Rv,w\rangle_{\Y}\vert\leq \langle Rv,v\rangle_{\Y}^{1/2}  \langle Rw,w\rangle_{\Y}^{1/2}
\end{equation}
then, we have
\begin{align*}
&\vert\langle R\tau^k,\sigma_h^k\rangle_{\Y}\vert \leq \dfrac{1}{4} \langle R\sigma_h^k,\sigma_h^k\rangle_{\Y} +  \langle R\tau^k,\tau^k\rangle_{\Y},\quad 
&\vert\langle R\dif\rho_h^k,\sigma_h^k\rangle_{\Y}\vert\leq \dfrac{1}{4} \langle R\sigma_h^k,\sigma_h^k\rangle_{\Y} +  \langle R\dif\rho_h^k,\dif\rho_h^k\rangle_{\Y}.
\end{align*}
On the other hand, by using the uniform continuity of the family of operators $A$, we can notice that
\[
\vert\langle A(t_k)\rho_h^k,\sigma_h^k\rangle_{\X}\vert \leq \dfrac{\alpha}{2} \Vert{\sigma_h^k}\Vert_{\V}^{2} + \dfrac{1}{2\alpha}\Vert A\Vert ^2 \Vert{\rho_h^k}\Vert_{\V}^{2}.
\]
Therefore, by replacing the previous inequalities in \eqref{pre1B} and using the fact that $R$ is a bounded operator and $X\subset Y$ is a continuous embedding, we deduce
\begin{align*}
\langle R\sigma_h^k,\sigma_h^k\rangle_{\Y}
-\langle R\sigma_h^{k-1},\sigma_h^{k-1}\rangle_{\Y}
&+\alpha\Delta t \|\sigma_h^k\|^2_{\X}\\
&\leq 
(1+2\lambda)\Delta t \langle R\sigma_h^k,\sigma_h^k\rangle_{\Y}
+ C\Delta t\left[
\Vert\tau^k\Vert_{Y}^2 
+ \Vert\dif\rho_h^k\Vert_{\X}^2
+ \Vert{\rho_h^k}\Vert_{\V}^{2}
\right].
\end{align*}
Hence, by summing over $k$, we obtain
\begin{align*}
&\langle R\sigma_h^n,\sigma_h^n\rangle_{\Y}
-\langle R\sigma_h^0,\sigma_h^0\rangle_{\Y}
+\alpha\Delta t\sum_{k=1}^n\|\sigma_h^k\|^2_{\V}\\
&\qquad\qquad\qquad\qquad
\leq(1+2\lambda)\Delta t\sum_{k=1}^n\langle R\sigma_h^k,\sigma_h^k\rangle_{\Y}
+C\Delta t\,\sum_{k=1}^n 
\left[
\Vert\tau^k\Vert_{Y}^2 
+ \Vert\dif\rho_h^k\Vert_{\X}^2
+ \Vert{\rho_h^k}\Vert_{\V}^{2}
\right].
\end{align*}
Then, if $\Delta t$ is small enough such that $(1+2\lambda)\Delta t \leq \frac12$, we have 
\begin{equation}\label{Cota1dg}
\begin{split}
&\dfrac12\langle R\sigma_h^n,\sigma_h^n\rangle_{\Y}+\alpha\Delta t\sum_{k=1}^n\|\sigma_h^k\|^2_{\V}\\
&\qquad\qquad
\leq
\langle R\sigma_h^0,\sigma_h^0\rangle_{\Y} +
(1+2\lambda)\Delta t\sum_{k=1}^{n-1}\langle R\sigma_h^k,\sigma_h^k\rangle_{\Y}+C\Delta t\,\sum_{k=1}^n 
\left[
\Vert\tau^k\Vert_{Y}^2 
+ \Vert\dif\rho_h^k\Vert_{\X}^2
+ \Vert{\rho_h^k}\Vert_{\V}^{2}
\right],
\end{split}
\end{equation}
which implies 
\[
\langle R\sigma_h^n,\sigma_h^n\rangle_{\Y}
\leq
2\langle R\sigma_h^0,\sigma_h^0\rangle_{\Y} +
2(1+2\lambda)\Delta t\sum_{k=1}^{n-1}\langle R\sigma_h^k,\sigma_h^k\rangle_{\Y}+C\Delta t\,\sum_{k=1}^n 
\left[
\Vert\tau^k\Vert_{Y}^2 
+ \Vert\dif\rho_h^k\Vert_{\X}^2
+ \Vert{\rho_h^k}\Vert_{\V}^{2}
\right].
\]
Therefore, by using the discrete Gronwall's Lemma (see, for instance, \cite[Lemma 1.4.2]{quarteronivalli}), we obtain
\begin{align*}
&\langle R\sigma_h^n,\sigma_h^n\rangle_{\Y}
\leq
C\left\{
\langle R\sigma_h^0,\sigma_h^0\rangle_{\Y} 
+\Delta t\,\sum_{k=1}^n 
\left[
\Vert\tau^k\Vert_{Y}^2 
+ \Vert\dif\rho_h^k\Vert_{\X}^2
+ \Vert{\rho_h^k}\Vert_{\X}^{2}
\right]\right\}.
\end{align*}
Hence, by using this inequality to estimate the second term in the right-hand term of \eqref{Cota1dg}, we deduce \eqref{desi-sigma}.

Next, we want to prove \eqref{estimasumaderivada-R} by assuming that each $A(t)$ is monotone and \eqref{Ap-bounded} holds true.
In fact, by taking $v=\dif\sigma_h^k\in \X_h$ in \eqref{indentityv}, we obtain 
\begin{align}\label{testingdifsigma-R}
\langle R \dif\sigma_h^k,\dif\sigma_h^k\rangle_{\Y}+\langle A(t_k)\sigma_h^k,\dif\sigma_h^k\rangle_{\X}
=
\langle R \tau^k,\dif\sigma_h^k\rangle_{\Y}-\langle R\dif \rho_h^k,\dif \sigma_h^k\rangle_{\Y}-\langle A(t_k)\rho_h^k,\dif\sigma_h^k\rangle_{\X }.
\end{align}
Now, since each operator $A(t)$ is monotone and self-adjoint, it follows
\begin{equation*}
\langle A(t_k)\dif \sigma_h^k,\sigma_h^k\rangle_{\X} 
\geq  \dfrac{1}{2\Delta t} \left\{\langle A(t_k)\sigma_h^k,\sigma_h^k\rangle_{\X}  
- \langle A(t_k)\sigma_h^{k-1},\sigma_h^{k-1}\rangle_{\X}
\right\},
\end{equation*}
and therefore
\begin{equation}\label{cotaoperatorAsigma-R}
\begin{split}
&\langle A(t_k)\sigma_h^k,\dif\sigma_h^k\rangle_{\X}\\
&\quad \geq
\frac{1}{2\Delta t}\left[ \langle A(t_k)\sigma_h^k,\sigma_h^k\rangle_{\X}-\langle A(t_{k-1})\sigma_h^{k-1},\sigma_h^{k-1}\rangle_{\X} \right]
-\frac{1}{2\Delta t}
\left\langle\left(\int_{t_{k-1}}^{t_k} A'(t)dt\right) \sigma_h^{k-1},\sigma_h^{k-1}\right\rangle_{\X}.
\end{split}
\end{equation}
On the other hand, a straightforward computation shows that
\begin{equation}\label{cotaoperatorArho-R}
\begin{split}
\langle A(t_k)\rho_h^k,\dif\sigma_h^k\rangle_{\X}&=
\frac{1}{\Delta t}\left[ \langle A(t_k)\rho_h^k,\sigma_h^k\rangle_{\X}-\langle A(t_{k-1})\rho_h^{k-1},\sigma_h^{k-1}\rangle_{\X} \right]-\langle A(t_k)\dif\rho_h^k,\sigma_h^{k-1}\rangle_{\X}\\
&\quad-\frac{1}{\Delta t}
\left\langle\left(\int_{t_{k-1}}^{t_k} A'(t)dt\right) \rho_h^{k-1},\sigma_h^{k-1}\right\rangle_{\X}.
\end{split}
\end{equation}
Hence, by using \eqref{cotaoperatorAsigma-R} and \eqref{cotaoperatorArho-R} in \eqref{testingdifsigma-R}, we have
\begin{equation}\nonumber
\begin{split}
\langle R \dif\sigma_h^k,\dif\sigma_h^k\rangle_{\Y}
&+\frac{1}{2\Delta t}\left[ \langle A(t_k)\sigma_h^k,\sigma_h^k\rangle_{\X}-\langle A(t_{k-1})\sigma_h^{k-1},\sigma_h^{k-1}\rangle_{\X} \right]
\\
&\leq 
\langle R \tau^k,\dif\sigma_h^k\rangle_{\Y}-\langle R\dif \rho_h^k,\dif \sigma_h^k\rangle_{\Y}
-\frac{1}{\Delta t}\left[ \langle A(t_k)\rho_h^k,\sigma_h^k\rangle_{\X}-\langle A(t_{k-1})\rho_h^{k-1},\sigma_h^{k-1}\rangle_{\X} \right]\\
&\hspace*{2cm} +\langle A(t_k)\dif\rho_h^k,\sigma_h^{k-1}\rangle_{\X}
+ \frac{1}{\Delta t}
\left\langle\left(\int_{t_{k-1}}^{t_k} A'(t)dt\right) (\rho_h^{k-1}+ \sigma_h^{k-1}),\sigma_h^{k-1}\right\rangle_{\X},
\end{split}
\end{equation}
then, recalling that the family of operators $A(t)$ is uniformly bounded and that the operator $R$ is also bounded, using \eqref{cauchy-R2} and \eqref{Ap-bounded}, it follows that
\begin{align*}
&\frac12\langle R\dif \sigma_h^k,\dif \sigma_h^k\rangle_{\Y}+
\frac{1}{2\Delta t}\left\{\langle A(t_k)\sigma_h^k,\sigma_h^k\rangle_{\X}-\langle A(t_{k-1})\sigma_h^{k-1},\sigma_h^{k-1}\rangle_{\X}\right\}\\
&\quad\leq 
-\frac{1}{\Delta t}\left\{\langle A(t_k)\rho_h^k,\sigma_h^k\rangle_{\X}-\langle A(t_{k-1})\rho_h^{k-1},\sigma_h^{k-1}\rangle_{\X}\right\}
+C \left\{\|\sigma_h^{k-1}\|_{\X}^2+ \|\tau^k\|_{\Y}^2+\|\dif \rho_h^k\|_{\X}^2+\|\rho_h^{k-1}\|_{\X}^2\right\},
\end{align*}
%
%
then, multiplying by $2\Delta t$, summing over $k$ and using the fact that 
$\langle A(t_n)\rho_h^n,\sigma_h^n\rangle_{\X}\leq \langle A(t_n)\rho_h^n,\rho_h^n\rangle_{\X}+\frac{1}{4}\langle A(t_n)\sigma_h^n,\sigma_h^n\rangle_{\X}$,  
we obtain
\begin{multline*}
\Delta t\sum_{k=1}^n\langle R \dif\sigma_h^k,\dif\sigma_h^k\rangle_{\Y}
+\frac{1}{2}\langle A(t_n)\sigma_h^n,\sigma_h^n\rangle_{\X}
\\ \qquad
\leq
\langle A(0)(2\rho_h^{0}+\sigma_h^{0}),\sigma_h^{0}\rangle_{\X}
+2\langle A(t_n)\rho_h^n,\rho_h^n\rangle_{\X}
+C \Delta t \sum_{k=1}^n\left\{\|\sigma_h^{k-1}\|_{\X}^2+ \|\tau^k\|_{\Y}^2+\|\dif \rho_h^k\|_{\X}^2+\|\rho_h^{k-1}\|_{\X}^2\right\}.
\end{multline*}
Finally, using \eqref{desi-sigma} to estimate the sum involving $\|\sigma_h^{k-1}\|_{\X}$ and recalling $A(t)$ is uniformly bounded and monotone,
we deduce \eqref{estimasumaderivada-R}.
\end{proof}

Now, we are in a position to prove the following error estimate.

\begin{theorem}\label{teo-deg-u}
If $u\in\H^1(0,T;\X)\cap H^2(0,T;\Y)$, then
there exists a constant $C>0$, independent of  $h$ and  $\Delta t$, such that
\begin{equation}\label{esti-eh} 
\begin{split}
&\max_{1 \leq n \leq N}\langle R(u(t_n)-u_h^n),u(t_n)-u_h^n\rangle_{\Y}
+\Delta t\,\sum_{n=1}^N\|u(t_n)-u_h^n\|^2_{\V}
\\
&\leq C\left\{ {\| u_0 -u_{0,h} \|_{\Y}^2} +\hspace*{-1mm}\max_{0 \leq n \leq N}\hspace*{-1mm}\left[ \inf_{v\in \X_h}{\|u(t_{n}) - v\|_{\X}^2}\right] 
\hspace*{-1mm}+\hspace*{-1mm}\int_0^T\hspace*{-2mm} \inf_{v\in\X_h} \| \partial_{t} u(t)-v \|_{\X}^2\, dt 
+ (\Delta t)^2\hspace*{-1mm}
\int_0^T \hspace*{-2mm}\| \partial_{tt} u(t)\|^2_{\Y}\,dt
\right\}.
\end{split}
\end{equation}
Furthermore, if $u_0\in\X$ and for each $t\in[0,T]$ the operator $A(t)$ is monotone and \eqref{Ap-bounded} holds true, then there exists a 
constant $C>0$, independent of  $h$ and  $\Delta t$, satisfying
\begin{equation}\label{esti-deh}
\begin{split}
 &\Delta t \sum_{k=1}^n\produ{R(\partial_{t}u(t_k)-\dif u_h^k),(\partial_{t}u(t_k)-\dif u_h^k)}_{\Y}
+\max_{1 \leq n \leq N}\langle A(t_n)(u(t_n)-u_h^n),u(t_n)-u_h^n\rangle_{\X}
\\
&
\leq C\left\{ {\| u_0 -u_{0,h} \|_{\X}^2} +\hspace*{-1mm}\max_{0 \leq n \leq N}\hspace*{-1mm}\left[ \inf_{v\in \X_h}{\|u(t_{n}) - v\|_{\X}^2}\right] 
\hspace*{-1mm}+\hspace*{-1mm}\int_0^T\hspace*{-2mm} \inf_{v\in\X_h} \| \partial_{t} u(t)-v \|_{\X}^2\, dt 
+ (\Delta t)^2\hspace*{-1mm}
\int_0^T \hspace*{-2mm}\| \partial_{tt} u(t)\|^2_{\Y}\,dt
\right\}.
\end{split}
\end{equation}
\end{theorem}
\begin{proof}
First of all, we notice that \eqref{infprodg} and \eqref{rhoysigmad} imply
\begin{equation}\label{ineq1-cea}
\|\rho_h^n\|_{\X}= \|\rho_h(t_n)\|_{\X}\leq C \inf_{z\in \X_h}\|u(t_n)-z\|_{\X}.
\end{equation}
Moreover, the regularity assumption about $u$ implies $\partial_{t}\Pi_h u(t)=\Pi_h(\partial_{t}u(t))$, and consequently
\[
\|\partial_{t}\rho_h(t)\|_{\X}\leq C \inf_{z\in \X_h}\|\partial_{t} u(t)-z\|_{\X}.
\]
Hence, it is easy to check that
\begin{align*}
\Delta t\sum_{k=1}^N \|\dif \rho_h^k \|^2_{\X} 
=\Delta t\sum_{k=1}^N\left\|\frac{1}{\Delta t}\int_{t_{k-1}}^{t_k} \hspace*{-2mm}\partial_t\rho_h(t)\dt\right\|_X^2
\leq\sum_{k=1}^N \int_{t_{k-1}}^{t_k} \hspace*{-2mm}\|\partial_t\rho_h(t)\|_X^2\dt
\leq  C \int_0^T \hspace*{-2mm} \inf_{v\in\V_h} \| \partial_{t} u(t)-v \|_{\V}^2\, dt.
\end{align*}
On the other hand, by combining a Taylor expansion with the Cauchy-Schwarz inequality, we obtain 
\begin{equation}\nonumber
\sum_{k=1}^N \| \tau^k \|^2_{\Y}
=\sum_{k=1}^N \left\|\frac{1}{\Delta t} \int_{t_{k-1}}^{t_{k}} (t_{k-1}-t)\partial_{tt}u(t)\dt\right\|_{Y}^2
\leq \Delta t \int_0^T\|\partial_{tt}u(t)\|_{\Y}^2 \dt.
\end{equation}
Now, by writing $\sigma_h^0=e_h^0-\rho_h^0$ and using the fact that $R$ is self-adjoint and  
monotone\footnote{Notice that if $R$ is self-adjoint and  monotone, we have $\langle R(v+w),(v+w)\rangle_{\Y}\leq 2\left[ \langle R(v),v\rangle_{\Y} + \langle R(w),w\rangle_{\Y}\right]$ for any $v,w\in\Y$.}, 
from the second equation of Problem~\ref{ppdc}, it follows that
\begin{equation}
\label{ineq4-cea}
\langle R\sigma_h^0,\sigma_h^0\rangle_{\Y}\leq 2\langle R(u_0-u_{0,h}),u_0-u_{0,h}\rangle_{\Y}+2\langle R\rho_h^0,\rho_h^0\rangle_{\Y}.
\end{equation}
By using inequalities \eqref{ineq1-cea}--\eqref{ineq4-cea} and Lemma~\ref{lemdg}, \eqref{esti-eh} follows from the fact that 
$u(t_n)-u_h^n=\rho_h^n+\sigma_h^n$ (see \eqref{splitE-d}) and the triangle inequality. 

Next, we need to deduce \eqref{esti-deh}. To this aim, we first recall that $\partial_{t}u(t_k)-\dif u_h^k = \left[\dif u(t_k)-\dif u_h^k\right]-\tau^k$,
then, by using \eqref{splitE-d} it follows 
$ \partial_{t}u(t_k)-\dif u_h^k = \left(\dif \rho_h^k + \dif \sigma_h^k\right) - \tau^k$.  Therefore, it is easy to obtain
\begin{equation*}
\produ{R(\partial_{t}u(t_k)-\dif u_h^k),\partial_{t}u(t_k)-\dif u_h^k}_{\Y}
\leq C \left[
\langle R\dif \sigma_h^k,\dif \sigma_h^k\rangle_{\Y} + \| {\dif \rho_h^k}\|_{\Y}^2 + \| {\tau^k}\|_{\Y}^2
\right].
\end{equation*}
Consequently, \eqref{esti-deh} follows by using \eqref{estimasumaderivada-R}, by proceeding as in the proof of \eqref{esti-eh} and noticing that
\begin{equation}\nonumber
 \Delta t\,\sum_{n=1}^N \inf_{v\in\X_h}{\|u(t_n)-v\|^2_{\X}}
 \leq  
 T 
 \max_{1 \leq n \leq N}\left[ \inf_{v\in \V_h}{\|u(t_{n}) - v\|_{\V}^2} \right]. 
\end{equation} 
\end{proof}
\section{Application to the eddy current problem}\label{aplicaciones-d}

The eddy current model is obtained by dropping the displacement currents from Maxwell equations  \cite[chapter 8]{bossavit}) and it provides a reasonable approximation to the solution of the full Maxwell system in the low frequency range (see \cite{AB}). This model is commonly used in many problems in science and industry: induction heating, electromagnetic braking, electric generation, etc (see \cite[Chapter 9]{alonsovallilibro}). The purpose for the eddy current problem is to determine the eddy currents induced a three-dimensional conducting domain $\ocA$ by a given time dependent compactly-supported current density $\Jn$. The eddy current problem can be read as follows.
\begin{problem}
\label{eddy-d}
Find the magnetic field $\hh:\R^3\times[0,T]\to\R^3$ and the electric field $\ee:\R^3\times[0,T]\to\R^3$ satisfying
\begin{align*}
\partial_t \left(\mu \hh\right)+ \curl \ee &= \cero,
\\
\curl \hh &= \jj + \sigma \ee,
\\
\dive (\varepsilon \ee) &= 0, 
\\
\dive (\mu \hh) &= 0, 
\end{align*}
where $\mu$, $\sigma$ and $\varepsilon$ represent the physical (scalar) parameters respectively called magnetic permeability, electric conductivity and electric permittivity.
\end{problem}

We assume that these parameters are piecewise smooth real valued functions satisfying:
\begin{align*}
&\varepsilon_{\max}\geq \varepsilon(\xn)\geq \varepsilon_{\min} > 0\ \quad\pct \text{ in } \ocA
\quad\text{and}\quad\varepsilon(\xn)= \varepsilon_{\min} \quad \pct\text{ in }\R^3\setminus\occA,
\\
&\sigma_{\max}\geq \sigma(\xn)\geq \sigma_{\min} > 0 \quad\pct \text{ in } \ocA
\quad\text{and}\quad\sigma(\xn)= 0 \ \quad \quad \pct\text{ in }\R^3\setminus\occA,
\\
&\mu_{\max}\geq \mu(\xn)\geq \mu_{\min} > 0 \quad\pct \text{ in } \ocA
\quad\text{and}\quad\mu(\xn)= \mu_{\min}  \quad \pct\text{ in }\R^3\setminus\occA.
\end{align*}
Different formulations for the eddy current model (\cite{zlamal, maccamy, reales}) can be analyzed as a degenerate parabolic problem of Section~\ref{degenerate} and the mathematical analysis of their numerical approximation by using finite element methods can be obtained with the theory performed in Sections \ref{discreto-d} and \ref{error-d}, however we only focus in the formulation studied in the first of that references. Zlamal \cite{zlamal} (see also \cite{zlamal-ad}) has proposed a solution of a particular case of the eddy current Problem~\ref{eddy-d} by solving the following two-dimensional degenerate parabolic problem, for a given data source $\jd:\R^2\times [0,T]\to \R$.
\begin{problem}
\label{two-d}
Find $u:\R^2\times[0,T]\to \R$ such that
\begin{equation}\label{vari-two-d}
\sigma\deriparc{u}{t}=\dive\left(\dfrac{1}{\mu}\grad u\right)+ \jd,
\end{equation}
where the physical parameters $\sigma$ and $\mu$ are independent of $x_3$.
\end{problem}
The following result shows the relationship between the eddy current Problem~\ref{eddy-d} and the degenerate parabolic equation Problem~\ref{two-d}. 
\begin{proposition}\label{rela-two-three}
If $u:\R^2\times[0,T]\to\R$ is an enough regular solution of Problem~\ref{two-d}  and the electric permittivity $\varepsilon$ is independent of $x_3$, then
\begin{align}\label{def-EH}
\Em:=\paren{0,0,-\deriparc{u}{t}}\quad\text{and}\quad\Hm:=\dfrac{1}{\mu}\paren{\deriparc{u}{x_2},-\deriparc{u}{x_1},0}
\end{align}
are solutions of problem Problem~\ref{eddy-d} with 
$\Jn:=\paren{0,0,\jd}$.
\end{proposition}
\begin{proof}
Let $u$ be a regular solution of Problem~\ref{two-d} and assume that $\Jn:=\paren{0,0,\jd}$. Let us define $\Em$ and $\Hm$ as in \eqref{def-EH}. 
Therefore, 
\begin{equation*}
\curl\Em
=\left(
-\dfrac{\partial}{\partial x_2}\left(\dfrac{\partial u}{\partial t} \right),\dfrac{\partial}{\partial x_1}\left(\dfrac{\partial u}{\partial t} \right),0\right)\\
=-\dfrac{\partial }{\partial t}(\mu\Hm),
\end{equation*}
and the first equation of Problem~\ref{eddy-d} follows. Furthermore, the second equation of Problem~\ref{two-d} is obtained by noticing that
\begin{equation*}
\curl\Hm=\left(0,0,-\dfrac{\partial}{\partial x_1}\left(\frac{1}{\mu}\dfrac{\partial u}{\partial x_1}\right) 
- \dfrac{\partial}{\partial x_2}\left(\frac{1}{\mu}\dfrac{\partial u}{\partial x_2}\right)\right)
=\left(0,0,-\dive\left(\frac{1}{\mu}\grad u\right)\right)
=\Jn + \sigma\Em.
\end{equation*} 
Next, by recalling that $u$ and $\varepsilon$  are independent of $x_3$, it follows the third equation of Problem~\ref{eddy-d}. Finally, the last equation of Problem~\ref{eddy-d} follows by using the regularity of $u$. 
\end{proof}
\subsection{Well-posedness for the eddy current formulation}
Let $\oA\subset\R^3$ be a  simply connected and bounded set containing  $\ocA$ and $\supp\Jn$, with $\Jn$ as in Proposition~\ref{rela-two-three}. 
In order to obtain a weak formulation for Problem~\ref{two-d}, we have to consider the projection of both sets $\oA$ and the conducting domain $\ocA$ onto the plane $x_1x_2$,
that will be denoted respectively as $\Omega$ and $\oc$.
Then, given $u_0\in\ldosoc$ and $\jd\in\ldosto$, by multiplying equation \eqref{vari-two-d} with $v\in\houno$ and integrating by parts over $\Omega$, we obtain the following weak formulation for the Problem~\ref{two-d}.
\begin{problem}
\label{two-dw1}
Find $u\in\ldoshouno$ such that
\begin{align*}
\dfrac{d}{dt}\int_{\oc} \sigma u v + \int_{\Omega} \dfrac{1}{\mu}\grad u\cdot \grad v &= \int_{\Omega} \jd v \quad\forall v\in\houno,\\
u(0)&=u_0 \quad\qquad \textrm{in }\oc.
\end{align*}
\end{problem}
The analysis of existence and uniqueness of solution for the previous problem is obtained by using Theorem~\ref{welld}. To this aim, in order to fit Problem~\ref{two-dw1} in the abstract structure of Problem~\ref{ppdc}, we have to define $X:=\houno$ and $Y:=\ldoso$, with their usual inner products. Then, we can easily deduce that these spaces satisfy the corresponding properties of Section~\ref{degenerate}. Furthermore, we define the operators $R:Y\to Y'$ and $A:X\to X'$ given by
\begin{align}
\left\langle Av,w\right\rangle_{\X}&:=\int_\Omega\dfrac1\mu\grad v\cdot\grad w\qquad\forall v,w\in X,\label{defA-d}\\
\left\langle Rv,w\right\rangle_{\Y}&:=\int_{\oc}\sigma v w\ \qquad\qquad\forall v,w\in Y.\label{defR-d}
\end{align}
We can notice that in this case the family of operators $A(t)$ in Problem~\ref{ppdc} is constant with respect of $t$.
Additionally, we need to define the function $f\in\L^2(0,T;X')$ given by
\begin{equation}\label{deff-d}
\left\langle f(t),v\right\rangle_{X}:=\int_\Omega\jd(t)v\qquad\forall v\in X.
\end{equation}
Finally, we should notice that the initial condition to Problem~\ref{two-dw1} is equivalent to $Ru(0)=Ru_0$ in $Y'$.

\begin{theorem}\label{well-eddy-two}
There exists a unique solution $u$ of Problem~\ref{two-dw1} satisfying
\begin{equation*}
\norm{u}_{\ldoshouno}\leq 
C\left\{
\norm{u_0}_{\ldosoc} + \norm{\jd}_{\ldosto}
\right\}.
\end{equation*}
\end{theorem}
\begin{proof}
The operator $R$ is clearly monotone and self-adjoint. 
Furthermore, the following G\r{a}rding-type inequality holds true for all $v\in X$:
\begin{equation}\label{garding-d}
\left\langle Rv,v\right\rangle_{\Y} + \left\langle Av,v\right\rangle_{\X} =\int_{\oc}\sigma\abs{v}^2 + \int_{\Omega}\dfrac{1}{\mu}\abs{\grad v}^2
\geq  \dfrac{1}{\mu_{\max}}\int_{\Omega}\abs{\grad v}^2 \geq  \dfrac{C_{\mathrm{P}}}{\mu_{\max}}\norm{v}_{\huno}^2,
\end{equation}
where $C_{\mathrm{P}}$ is the positive constant given by the Poincar\'e inequality in $\houno$. Consequently, Theorem~\ref{welld} shows that 
Problem~\ref{two-dw1} has at least a solution. 
Moreover, since the family of operators $A$ is independent of time, it is trivially a regular family and consequently the solution $u$ of 
Problem~\ref{two-dw1} is unique.
Finally, by using \eqref{dependencia-d} and noticing that
\[
\left\langle Ru_0,u_0\right\rangle_{\Y}=\int_{\oc}\sigma\abs{u_0}^2\leq \sigma_{\max}\norm{u_0}_{\ldosoc}^2,
\]
we conclude the proof.
\end{proof}
\begin{remark} It is easy to see that 
\begin{equation*}
\sigma\partial_t u - \dive\left(\frac1\mu\nabla u\right)= \jd \qquad\textrm{in }\ldoshounop,
\end{equation*}
consequently $u\vert_{\oc}$ belongs to the space $W^{1,2}(0,T;\hunooc,\hunooc')$. 
\end{remark}

\subsection{Error estimates for the fully-discrete degenerate formulation}
The fully-discrete approximation for the degenerate Problem~\ref{two-dw1} is obtained by using a finite element subspaces to define $X_h$ which is the corresponding family of finite dimensional subspaces of $X$ (see Section~\ref{discreto-d}). To this aim, in what follows we assume that $\Omega$ and $\oc$ are Lipschitz polygonal. Let $\set{\mathcal{T}_h}_h$ be a regular family of triangles meshes of $\Omega$ such that each element $K\in\mathcal{T}_h$ is contained either in $\overline{\Omega}_{\mathrm{c}}$ or in $\overline{\Omega}_{\mathrm{d}}:=\overline{\Omega\setminus\occ}$. As usual, $h$ stands for the largest diameter of the triangles $K$ in $\mathcal{T}_h$. 

We define $X_h$ using the standard Lagrange finite element subspace of $\houno$, \textit{i.e.},
\[
X_h:=\left\{ v_h\in C^0(\overline\Omega): v\vert_{K}\in \mathbb{P}_1(K)\right\}\cap \houno,
\]
where $C^0(\overline\Omega)$ is the space of scalar continuous functions defined on $\overline\Omega$ and $\mathbb{P}_1$ is the set of polynomials of degree not greater than $1$. Then, the fully-discrete approximation for the degenerate parabolic formulation is given by Problem~\ref{ppdd}, by using the notation \eqref{defA-d}--\eqref{deff-d}. More precisely, Given $u_{0,h}\in X_h$ an approximation of $u_0$, the fully-discrete approximation of Problem~\ref{two-dw1} can be read as follows.
\begin{problem}
\label{fd-eddy-d}
Find $u_h^n\in\X_h$, $n=1,\dots, N$, such that
\begin{align*}
\int_{\oc} \sigma \left(\dfrac{\uhn-\uhnm}{\Delta t}\right) v + \int_{\Omega} \dfrac{1}{\mu}\grad\uhn\cdot \grad v &= \int_{\Omega} \jd(t_n) v
\qquad \forall v\in \X_h,\\
u_h^0&=u_{0,h}.
\end{align*}
\end{problem}

Thus, by using \eqref{garding-d}, the existence and uniqueness of solution $u_h^n\in\X_h$, $n=1,\dots, N$, of the fully-discrete problem is guaranteed by Theorem~\ref{welldh}
for a small enough time-step. Moreover, by noticing that in this case we have
\[
\produ{R(\partial_{t}u(t_k)-\dif u_h^k),\partial_{t}u(t_k)-\dif u_h^k}_{\Y} = \int_{\oc}\sigma\norm{\partial_{t}u(t_k)-\dif u_h^k}_{\ldosoc}^2,
\]
we obtain the following result about the error estimates for the fully-discrete approximation Problem~\ref{fd-eddy-d} of the degenerate parabolic Problem~\ref{two-dw1}, which is a direct consequence of Theorem~\ref{teo-deg-u}.

\begin{theorem}\label{cea-eddy-d}
Let $u\in\ldoshouno$ be the solution of the eddy current Problem~\ref{two-dw1} and $u_h^n\in\X_h$, $n=1,\dots, N$, the fully-discrete solution of Problem~\ref{fd-eddy-d}.
If $u_0\in\houno$ and $u\in \H^1(0,T;\houno)\cap\H^2(0,T;\ldoso)$ then there exists a constant $C>0$, independent of $h$ and $\Delta t$, such that
\begin{align*}
&\max_{1\leq n\leq N}\|u(t_n) - \uhn\|_{\sigma,\oc}^2 
+ \Delta t\sum_{n=1}^{N}\|u(t_n) - \uhn\|_{\houno}^2
+ \Delta t\sum_{n=1}^{N}\norm{\partial_{t}u(t_n)-\dif u_h^n}_{\sigma,\oc}^2
\\
&\quad\quad
\leq C \left\{ 
\|u_0 -u_{0,h} \|_{\houno}^2
+\max_{0 \leq n \leq N}\left[ \inf_{v\in \V_h}{\|u(t_{n}) - v\|_{\houno}^2}\right] 
+\int_0^T \inf_{v\in\V_h} \| \partial_{t} u(t)-v \|_{\houno}^2\,dt 
\right.\\[1ex]
&
\qquad\qquad\quad
\left.
+ (\Delta t)^2 \int_0^T \| \partial_{tt} u(t) dt\|^2_{\ldoso} \right\},
\end{align*}
where 
$\norm{w}_{\sigma,\oc}^2:=\displaystyle\int_{\oc}\sigma\vert w\vert^2$.
\end{theorem}
Finally, to obtain the asymptotic error estimate, we need to consider the Sobolev space $\mathrm{H}^{1+s}(\Omega)$ for $0<s\leq1$. 
It is well known that the Lagrange interpolant $\mathcal{L}_h v\in X_h$ is well defined for all $v\in\mathrm{H}^{1+s}(\Omega)\cap\houno$
and satisfies the following estimate (see, for instance, \cite{ciarlet})
\begin{equation}\label{bimbo-d}
\norm{v-\mathcal{L}_h v}_{\houno} \leq C
h^{s} \norm{v}_{\mathrm{H}^{1+s}(\Omega)}
\qquad \forall v\in\mathrm{H}^{1+s}(\Omega)\cap\houno. 
\end{equation}
Consequently, we have the following result which shows the asymptotic convergence of the fully-discrete approximation.
\begin{corollary}\label{coroconvucd-de} 
If $u_0\in\houno$ and $u\in \H^1(0,T;\houno\cap\mathrm{H}^{1+s}(\Omega))\cap \H^2(0,T;\ldoso)$ for $0<s\leq1$, there exists a constant $C>0$ independent of $h$ and $\Delta t$, such that
\begin{align*}
&\max_{1\leq n\leq N}\| u(t_n)-\uhn\|_{\sigma,\oc}^2 
+ \Delta t\sum_{n=1}^{N}\|u(t_n)-\uhn\|_{\houno}^2
+ \Delta t\sum_{n=1}^{N}\norm{\partial_{t}u(t_n)-\dif u_h^n}_{\sigma,\oc}^2
\\
&\qquad\qquad
\leq C \left\{
\|u_0 -u_{0,h} \|_{\houno}^2
+h^{2s}\left[
\max_{1\leq n\leq N}
\|u(t_n)\|_{\mathrm{H}^{1+s}(\Omega)}^2
+\|\partial_t u\|_{\L^2(0,T;\mathrm{H}^{1+s}(\Omega))}^2
\right]\right.\\
&\quad
\left.
\qquad\quad
\phantom{\max_{1\leq n\leq N}}
+(\Delta t)^2\|\partial_{tt}u\|_{\L^2(0,T;\ldoso^3)}^2
\right\}.
\end{align*}
Moreover, if $u_0\in\houno\cap\mathrm{H}^{1+s}(\Omega)$, for $0<s\leq1$ and $u_{0,h}=\mathcal{L}_h u_0$ then
\[
\max_{1\leq n\leq N}\| u(t_n)-\uhn\|_{\sigma,\oc}^2 
+ \Delta t\sum_{n=1}^{N}\|u(t_n)-\uhn\|_{\houno}^2
+ \Delta t\sum_{n=1}^{N}\norm{\partial_{t}u(t_n)-\dif u_h^n}_{\sigma,\oc}^2
=\mathcal{O}(h^{2s}+(\Delta t)^2).
\]
\end{corollary}
\begin{proof}
It is a direct consequence of Theorem~\ref{cea-eddy-d} and the interpolation error estimate \eqref{bimbo-d}.
\end{proof}

\begin{remark}\label{ErrorFisicas}
The previous result shows that the fully-discrete approximation Problem~\ref{fd-eddy-d} provides  a suitable approximation for the physical variables of the eddy current problem at each time $t_n$, namely the electric field $\ee(t_n)$ in the three-dimensional conducting  
domain $\ocA$ and the magnetic field $\Hn(t_n)$ in the three-dimensional computational domain $\oA$. 
More precisely, we can use the relationship \eqref{def-EH}, to define 
\[
\ee(t_n):= (0,0,-\partial_tu(t_n)) \quad\textrm{in }\ocA,\qquad 
\Hn(t_n):=\frac1\mu\left(\dfrac{\partial u}{\partial x_2}(t_n),-\dfrac{\partial u}{\partial x_1}(t_n),0\right)
 \quad\textrm{in }\oA,
\]
for any $n=1,\ldots,N$, and 
propose the following approximations
\begin{equation}\nonumber
\ee(t_n)\approx\ee_h^n:=(0,0,-\dif\uhn) \qquad\textrm{in }\ocA,
\end{equation}
and
\begin{equation}\nonumber
\Hn(t_n)\approx\Hn_h^n:=\frac1\mu\left(\dfrac{\partial\uhn}{\partial x_2},-\dfrac{\partial\uhn}{\partial x_1},0\right) 
\qquad\textrm{in }\oA.
\end{equation}
Consequently, by using Corollary~\ref{coroconvucd-de}, we deduce the following quasi-optimal error estimates
\begin{equation}\nonumber
\Delta t\sum_{n=1}^{N}\|\ee(t_n)-\ee_h^n\|_{\sigma,\ocA}^2
+ \Delta t\sum_{n=1}^{N}\norm{\Hn(t_n)- \Hn_h^n}_{\mu,\oA}^2
\leq 
\|u_0 -u_{0,h} \|_{\houno}^2
+
C\left[
h^{2s}
+(\Delta t)^2
\right],
\end{equation}
where 
$\norm{\bw}_{\mu,\oA}^2:=\displaystyle\int_{\oA}\frac1\mu\vert \bw\vert^2$.
\end{remark}
\subsection{Numerical results}
In this subsection we present some numerical results obtained with a MATLAB code which implements the numerical method described in Problem~\ref{fd-eddy-d},  to illustrate the convergence with respect to the discretization parameters. To this end, we describe the results obtained for a test problem with a known analytical solution.

\begin{figure}[!htb]
	\begin{center}
		\includegraphics*[scale=0.3]{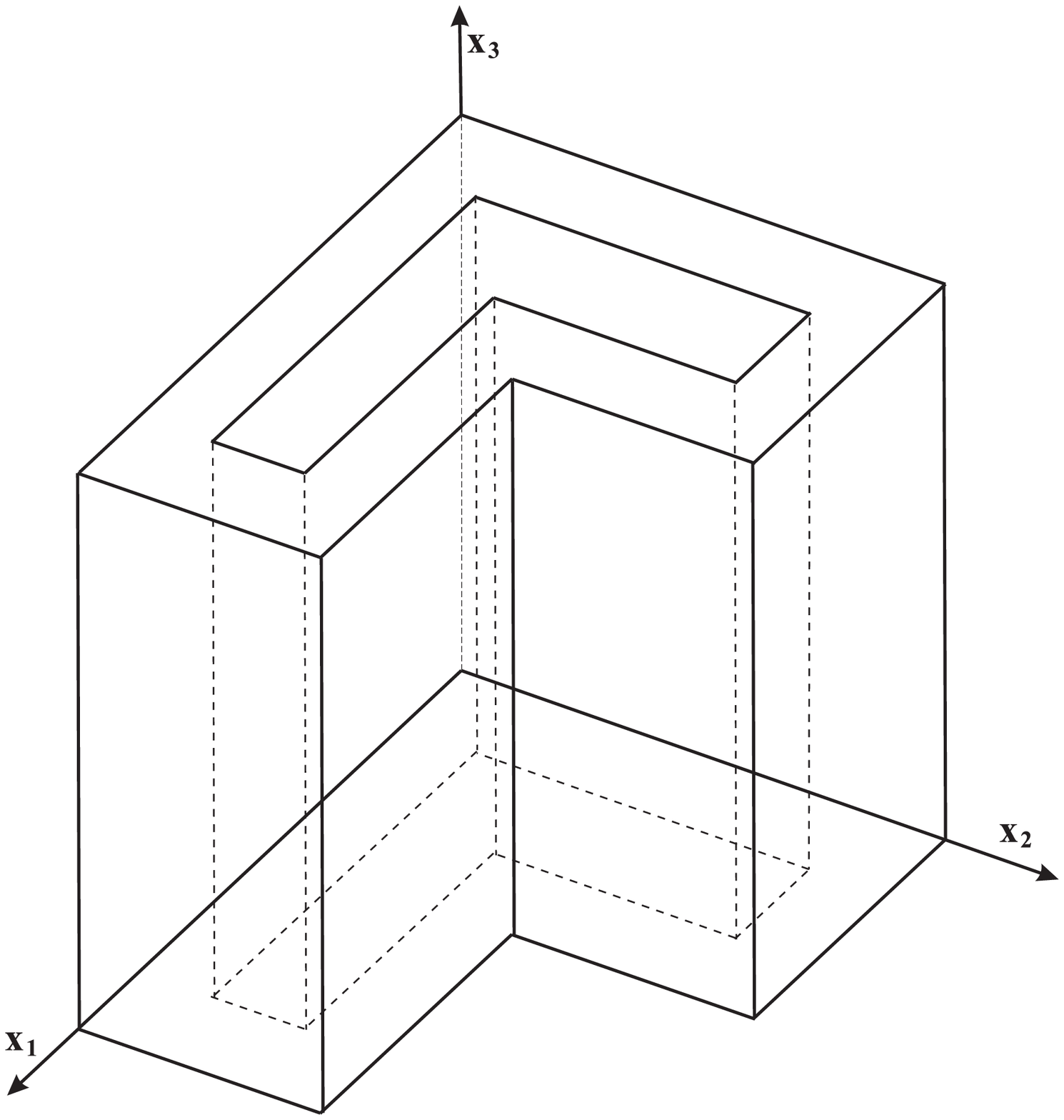}\hspace*{2cm}\includegraphics*[scale=0.45]{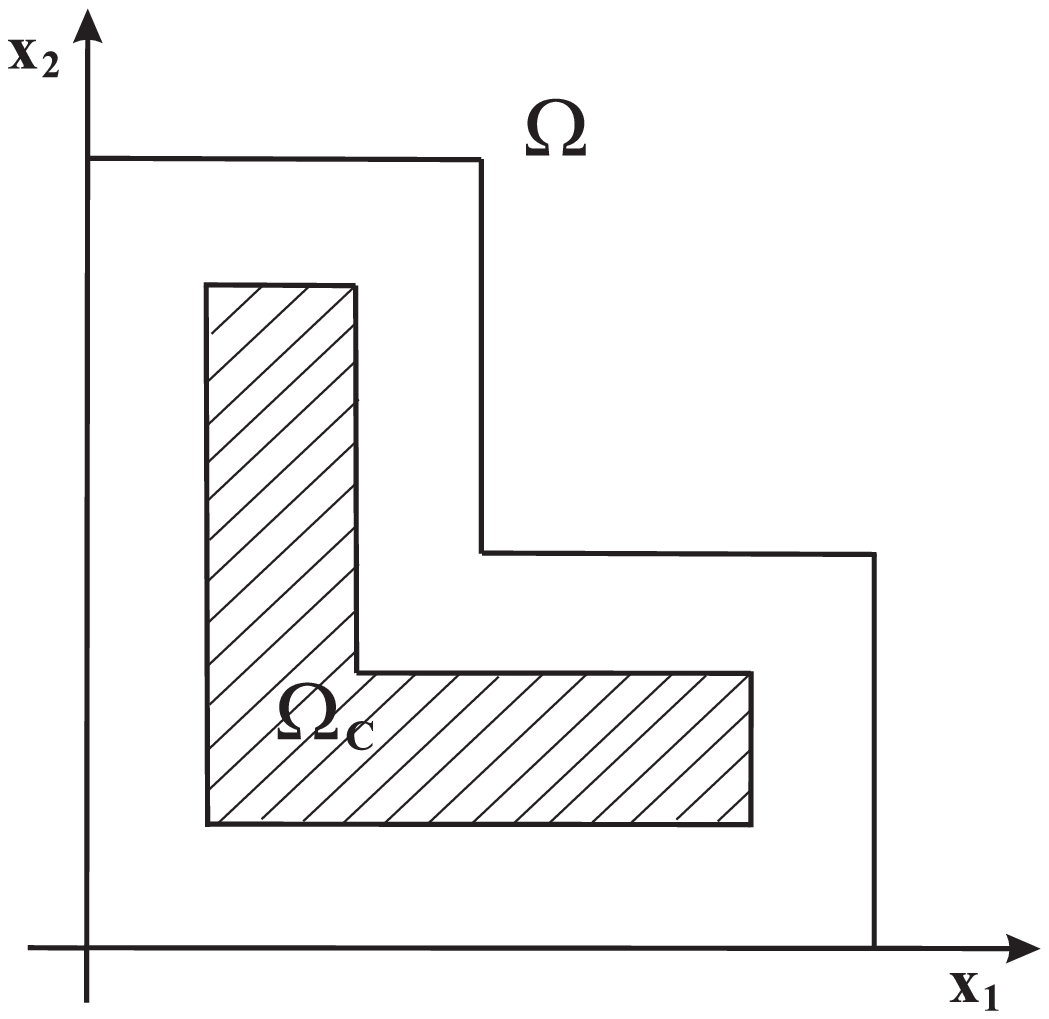}
		\caption{Sketch of the domain 3D (left) and 2D (right).}
		\label{domain}
	\end{center}
\end{figure}

We consider $\oA$ with $\ocA$ and their respective projection onto the plane $x_1x_2$, $\Omega$ and $\oc$ (see Figure~\ref{domain}) and $T=1$. The right hand side $J_d$, is chosen so that
\[
u(x_1,x_2,t)=e^{-5\pi t} \sin (\pi x_1) \sin(\pi x_2),
\]
is the solution to Problem~\ref{two-d} in $\Omega$ with boundary condition $u=0$ on $\partial\Omega$. Notice that $u$ is also solution of Problem~\ref{two-dw1} with $u_0(x_1,x_2)=\sin (\pi x_1) \sin(\pi x_2)$ where, in particular $u_0\in\H^1_0(\Omega)\cap\H^2(\Omega)$. We have taken  $\unit[\mu=\mu_0=4\pi\times10^{-7}]{{Hm}^{-1}}$, $\unit[\sigma=\sigma=10^{6}]{(\Omega m)^{-1}}$ in $\oc$, the magnetic permeability and electric conductivity of vacuum, respectively. The numerical method has been applied with several successively refined meshes and time-steps. The computed approximate solutions have been compared with the analytical one, by calculating the relative percentage error in time-discrete norms from Corollary \ref{coroconvucd-de}. More accurately, thanks to Proposition~\ref{rela-two-three} and Remark~\ref{ErrorFisicas}, we have compute the relative percentage error for the physical variables of interest, the magnetic field and the electric field in the conductor domain, namely
\begin{equation*}
100\,\frac{\Delta t\sum_{n=1}^{N}\norm{\Hn(t_n)- \Hn_h^n}_{\mu,\oA}^2}{\Delta t\sum_{n=1}^{N}\norm{\Hn(t_n)}_{\mu,\oA}^2}
\quad\text{and}\quad
100\,\dfrac{\Delta t\sum_{n=1}^{N}\|\ee(t_n)-\ee_h^n\|_{\sigma,\ocA}^2}{\Delta t\sum_{n=1}^{N}\|\ee(t_n)\|_{\sigma,\ocA}^2},
\end{equation*}
which are time-discrete forms of the errors in $\L^{2}(0,T;\L^2(\oA))$ and $\L^{2}(0,T;\L^2(\ocA))$ norms, respectively.

The Table~\ref{TablaH} shows the relative errors for $\Hn$ in the $\L^{2}(0,T;\L^2(\oA))$-norm, namely the relative errors for $u$ in the $\L^2(0,T;\H^1_0(\Omega))$-norm. We notice that by taking a small enough time-step $\Delta t$, we can observe the behavior of the error with  respect to the space discretization (see the row corresponding to $\Delta t/64$). On the other hand, by considering a small enough mesh-size $h$, we can check the order convergence with respect $\Delta t$ (see the first entries of the column corresponding to $h/64$). Hence, we conclude an order the convergence $\mathcal{O}(h+\Delta t)$ for $\Hn$, which confirm the theoretical results given in Remark~\ref{ErrorFisicas}, proved in Corollary \ref{coroconvucd-de}.

\begin{table}[!htb]
\begin{center}
\begin{tabular}{lccccccccc}
\hline
 & $h$ &$h/2$ & $h/4$ & $h/8$ &$h/16$& $h/32$&$h/64$\\
\hline
$\Delta t$ & \fbox{41.3685}& 22.1296& 12.8925 & 9.1603 &  7.9516& 7.6190 & 7.5335\\
$\Delta t/2$  & 41.3088& \fbox{21.4624}& 11.4341 & 6.8342 & 5.0574& 4.5040& 4.3546& \\
$\Delta t/4$  &41.4454& 21.3041& \fbox{10.9212} & 5.8293 & 3.5396& 2.6751 &2.4108 \\
$\Delta t/8$  &41.5820& 21.3044& 10.7883 & \fbox{5.5072} & 2.9460&1.845&1.3784 \\
$\Delta t/16$ & 41.6723& 21.3307& 10.7652 & 5.4225 &\fbox{2.7648} &1.4813&0.9115 \\
$\Delta t/32$ & 41.7237& 21.3514& 10.7663 & 5.4038 & 2.7172& \fbox{1.3851}&0.7428 \\
$\Delta t/64$ & 41.7511& 21.3637& 10.7702& 5.4008& 2.7059&1.3599& \fbox{0.6932} \\
\hline
\end{tabular}
\caption{Percentage errors for $\Hn$ in the $\L^{2}(0,T;\L^2(\oA))$-norm, with $h=0.3687$ and $\Delta t=0.025$.} 
\label{TablaH}
\end{center}
\end{table}

The Table~\ref{TablaDerU} shows the relative errors for $\ee$ in $\L^{2}(0,T;\L^2(\ocA))$, namely the relative errors  $\partial_t u$ in the $\L^2(0,T;\L^2(\oc))$-norm. We proceed as above, now we can see an order the convergence $\mathcal{O}( h^2+\Delta t)$ (see the row corresponding to $\Delta t/512$ and the column corresponding to $h/16$), in spite of the fact that only a linear order of convergence in $h$ has been proved above. Hence, we have obtained the theoretical results proved in Corollary \ref{coroconvucd-de}, too.

\begin{table}[!htb]
\begin{center}
\begin{tabular}{lcccccccc}
\hline
& $h$ &$h/2$ & $h/4$ & $h/8$ &$h/16$\\
\hline
$\Delta t$ & \fbox{26.3489}& 23.9703 & 23.6728& 23.6232 & 23.6127\\
$\Delta t/2$   & 17.2551&13.4472& 13.1275 & 13.1028 & 13.1006 \\
$\Delta t/4$ & 13.7947&\fbox{7.5263} & 6.9433 & 6.9188 & 6.9213\\
$\Delta t/8$ & 13.2102& 4.8159&3.6233& 3.5566 & 3.5592 \\
$\Delta t/16$ &13.3954& 3.9628&\fbox{1.9873}&1.8078&1.8042 \\
$\Delta t/32$ & 13.6309& 3.8427& 1.3093 &  0.9290 & 0.9082\\
$\Delta t/64$ & 13.7873& 3.8923& 1.1142 &\fbox{0.5144} &0.4574\\
$\Delta t/128$ & 13.8756& 3.9494& 1.0886 &0.3501 &0.2352\\
$\Delta t/256$ & 13.9223& 3.9870& 1.0992 &0.3049 &\fbox{0.1323}\\
$\Delta t/512$ & 13.9463& 4.0081& 1.1111 &0.2992 &0.0927\\
\hline
\end{tabular}
\caption{Percentage errors for  $\ee$ in the $\L^2(0,T;\L^2(\oc))$-norm, with $h=0.3687$ and $\Delta t=0.025$.} 
\label{TablaDerU}
\end{center}
\end{table}

Figure~\ref{ErroresRel_HE} shows log-log plots of the error of $\Hn$ (left) and $\ee$ (right) versus number of degrees of freedom (d.o.f). To report this we have been values of $\Delta t$ proportional to $h$ (see the values within boxes in Table \ref{TablaH}) and $\Delta t$ proportional to $h^2$ (see the values within boxes in Table~\ref{TablaDerU}), respectively.  The slopes of the curves clearly show an order of convergence $\mathcal{O}(h+\Delta t)$ and $\mathcal{O}(h^2+\Delta t)$, respectively.

\begin{figure}[ht!]
	\begin{center}
		\includegraphics*[width=7cm]{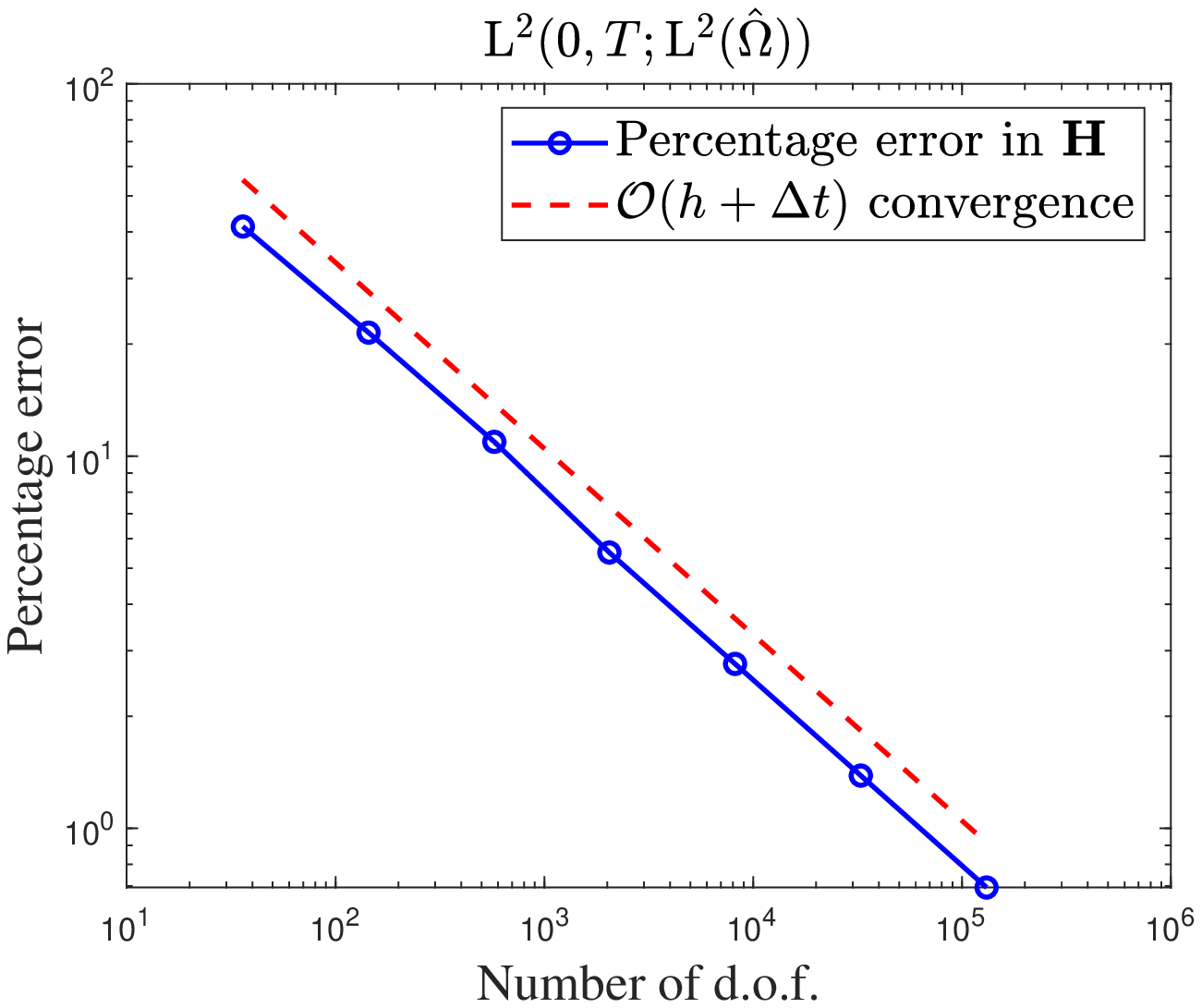}\includegraphics*[width=7cm]{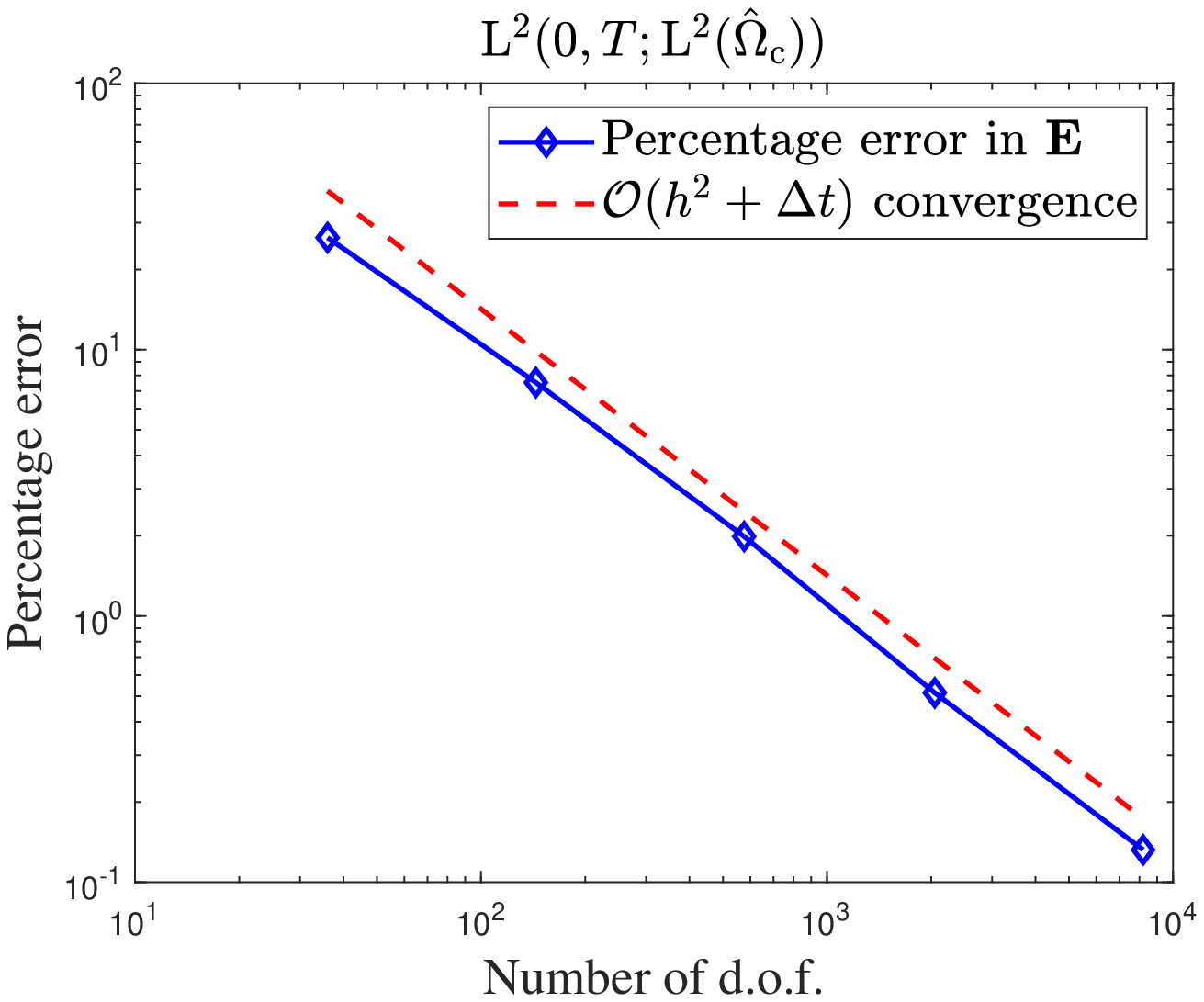}
		\caption{Percentage discretization error curves for $\Hn$ (left) and $\ee$ (right)	versus number of d.o.f. (log-log scale).}
		\label{ErroresRel_HE}
	\end{center}
\end{figure}

\textbf{Acknowledgments}\\
Thanks to Colciencias.

\textbf{Funding}\\
This work was partially supported by Colciencias through the 727 call, University of Cauca through project VRI ID 5243 and by Universidad Nacional de Colombia through Hermes project $46332$.

\textbf{Availability of data and materials}\\
Not applicable.

\textbf{Competing interests}\\
The authors declare that they have no competing interests.

\textbf{Authors contributions}\\
The authors declare that the work was realized in collaboration with the same responsibility. All authors read and
approved the final manuscript.

\end{document}